\documentclass[11pt]{article}
\usepackage[utf8]{inputenc}
\usepackage{amssymb,amsmath,amsfonts,amsthm,mathrsfs}
\usepackage{graphicx,graphics,psfrag,epsfig,cancel}
\usepackage[colorlinks=true, pdfstartview=FitV, linkcolor=blue, citecolor=blue, urlcolor=blue]{hyperref}
\usepackage[numbers,comma,square,sort&compress]{natbib}
\usepackage{dsfont}

\usepackage{caption,subfig,xcolor}
\usepackage{tikz,tkz-tab}
\usetikzlibrary{calc}

\usetikzlibrary{arrows,shapes,positioning}
\usetikzlibrary{decorations.markings}
\tikzstyle arrowstyle=[scale=1.25]
\tikzstyle directed=[postaction={decorate,decoration={markings,
    mark=at position .6 with {\arrow[arrowstyle]{stealth}}}}]
\tikzstyle reverse directed=[postaction={decorate,decoration={markings,
    mark=at position .5 with {\arrowreversed[arrowstyle]{stealth};}}}]

\usepackage{geometry}
\newtheorem{assumption}{Assumption}

\newtheorem{lemma}{Lemma}

\newcommand{\N}{{\mathbb N}}
\newcommand{\Z}{{\mathbb Z}}
\newcommand{\R}{{\mathbb R}}

\newcommand{\cercle}{{\mathbb S}^1}

\newcommand{\bqs}{\begin{equation*}}
\newcommand{\eqs}{\end{equation*}}
\newcommand{\bqq}{\begin{equation}}
\newcommand{\eqq}{\end{equation}}

\newcommand{\mbi}{\mathbf{i}}

\title{Amplification of numerical wave packets \\
for transport equations with two boundaries}

\author{Romain {\sc Bonnet-Eymard}\thanks{\'Ecole Normale Supérieure Paris-Saclay, 4 avenue des Sciences, 91190, Gif-Sur-Yvette, France. 
Email : {\tt rbonnete@ens-paris-saclay.fr}}, Jean-Fran\c{c}ois {\sc Coulombel} \& Gr\'egory {\sc Faye}\thanks{Institut de Math\'ematiques 
de Toulouse - UMR 5219, Universit\'e de Toulouse ; CNRS, Universit\'e Paul Sabatier, 118 route de Narbonne, 31062 Toulouse Cedex 9 , France. 
Research of J.-F. C. was supported by ANR project HEAD under grant agreement ANR-24-CE40-3260. G.F. acknowledges support from Labex 
CIMI under grant agreement ANR-11-LABX-0040, from ANR project Indyana under grant agreement ANR-21-CE40-0008 and from ANR project 
HEAD under grant agreement ANR-24-CE40-3260. Emails: {\tt jean-francois.coulombel@math.univ-toulouse.fr}, 
{\tt gregory.faye@math.univ-toulouse.fr}}}
\date{\today}

\begin{document}

\maketitle

\begin{abstract}
The purpose of this note is to investigate the coupling of Dirichlet and Neumann numerical boundary conditions for the transport equation set on 
an interval. When one starts with a stable finite difference scheme on the lattice $\Z$ and each numerical boundary condition is taken separately 
with the Neumann extrapolation condition at the \emph{outflow} boundary, the corresponding numerical semigroup on a \emph{half-line} is known 
to be bounded. It is also known that the coupling of such numerical boundary conditions on a \emph{compact interval} yields a \emph{stable} 
approximation, even though large time exponentially growing modes may occur. We review the different stability estimates associated with these 
numerical boundary conditions and give explicit examples of such exponential growth phenomena for finite difference schemes with ``small'' stencils. 
This provides numerical evidence for the optimality of some stability estimates on the interval.
\end{abstract}

\section{Introduction}
\label{section1}

We consider the following transport problem. Given a positive velocity $a>0$ and an interval length $L>0$, we consider the transport equation at 
velocity $a$ on the interval $(0,L)$ with homogeneous Dirichlet condition at the incoming boundary (that is, at $x=0$ here since $a$ is positive):
\begin{equation}
\label{transport}
\begin{cases}
\partial_t u +a \, \partial_x u \, = \, 0 \, ,& t \ge 0 \, ,\, x \in (0,L) \, ,\\
u|_{t=0} \, = \, f \, ,& x \in (0,L) \, ,\\
u|_{x=0} \, = \, 0 \, ,& t \ge 0 \, .
\end{cases}
\end{equation}
It is understood that $f$ is a given real valued function on $(0,L)$ that is, say, at least continuous on the closed interval $[0,L]$ with $f(0)=0$ so 
that the compatibility condition at the corner $t=x=0$ is satisfied. Extending $f$ by $0$ on the set $\R^-$ of negative real numbers, the solution 
to \eqref{transport} is given by the formula:
\begin{equation}
\label{solution}
\forall \, t \ge 0 \, ,\quad \forall \, x \in (0,L) \, ,\quad u(t,x) \, = \, f(x-a \, t) \, .
\end{equation}
In particular, the solution $u(t,\cdot)$ vanishes on $(0,L)$ for $t \ge L/a$.

The goal of this note is to investigate the numerical counterpart of this rather trivial problem. We do not aim at the most general framework and 
therefore make several simplifying assumptions. First of all, we consider once and for all a fixed parameter $\lambda>0$. Given the interval 
length $L>0$, we consider an integer $J \ge 1$ that is meant to be \emph{large}, and we define the \emph{space step} $\Delta x :=L/(J+1)$. 
The grid points are labeled as $x_j:=j \, \Delta x$ for any $j \in \Z$, in such a way that the interval $(0,L)$ is divided into the cells $(x_j,x_{j+1})$ 
with $j=0,\dots,J$. The \emph{time step} $\Delta t$ is then defined as $\Delta t:=\lambda \, \Delta x$. For a given integer $n \in \N$ and 
$j \in \{ 0,\dots,J \}$, we let $u_j^n$ denote the approximation of the solution to \eqref{transport} within the cell $(x_j,x_{j+1}) \times (n \Delta t,
(n+1)\Delta t)$. The considered numerical scheme will update the vector $(u_0^n,\dots,u_J^n)$ into a new vector $(u_0^{n+1},\dots,u_J^{n+1})$ 
for each $n \in \N$.

For convenience, we introduce the following notation for the discrete difference operators with respect to the spatial index $j$. Given a collection 
of three real numbers $u_{j-1},u_j,u_{j+1}$, the discrete derivatives $(D_-u)_j$ and $(D_+u)_j$ at the index $j$ are defined as:
$$
(D_-u)_j \, := \, u_j-u_{j-1} \, ,\quad (D_+u)_j \, := \, u_{j+1}-u_j \, ,
$$
and we view $D_-,D_+$ as operators acting on either sequences or vectors meaning that we allow ourselves to iterate them. For simplicity, we 
shall omit most of the time to write the brackets and simply use the notation $D_-u_j$ or $D_+u_j$. Such notation is used, of course, at any 
integer $j$ for which the action of the operator makes sense. As an example, we have:
$$
D_-^2u_j \, = \, u_j-2 \, u_{j-1}+u_{j-2} \, ,
$$
and larger powers of $D_-$ or $D_+$ can be computed by making appeal to binomial coefficients.

\paragraph{The considered numerical scheme.}

We consider two integers $r,p \in \N$ and some real coefficients $a_{-r},\dots,a_p$ with $a_{-r} \neq 0$ and $a_p \neq 0$. These coefficients are 
assumed to depend only on $\lambda$ and $a$. They should satisfy the \emph{consistency} conditions:
\begin{equation}
\label{consistency}
\sum_{\ell=-r}^p a_\ell \, = \, 1 \, ,\quad \sum_{\ell=-r}^p \ell \, a_\ell \, = \, - \lambda \, a \, .
\end{equation}
The numerical scheme in the \emph{interior} domain reads:
\begin{equation}
\label{schema-int}
\forall \, j=0,\dots,J \, ,\quad u_j^{n+1} \, = \, \sum_{\ell=-r}^p \, a_\ell \, u_{j+\ell}^n \, .
\end{equation}
It is understood that $r$ and $p$ are fixed while the number $J$ of cells may get arbitrarily large. The update \eqref{schema-int} from the discrete 
time $n$ to the following time level $n+1$ is possible only if we have the \emph{ghost cell} values $u_{-r}^n,\dots,u_{-1}^n$ and $u_{J+1}^n,\dots,
u_{J+p}^n$ at our disposal. The ghost cell values on the left of the interval are meant to provide for approximations of the trace of the solution so 
it seems reasonable, in view of \eqref{transport}, to impose the following homogeneous Dirichlet condition:
\begin{equation}
\label{schema-Dirichlet}
\forall \, n \in \N \, ,\quad \forall \, \mu=-r,\dots,-1 \, ,\quad u_\mu^n \, = \, 0 \, .
\end{equation}
This is not the only option but it will be the one we choose here for the sake of simplicity. Prescribing numerical boundary conditions on the right 
of the interval is a little more tricky since the continuous problem \eqref{transport} does not impose anything at first glance. We follow here the 
extrapolation procedure that has been analyzed in \cite{Kreiss-proc,Goldberg,CL-KRM} and other works. We thus consider a fixed integer $k \in \N$ 
and define the ghost cell values $u_{J+1}^n,\dots,u_{J+p}^n$ by imposing:
\begin{equation}
\label{schema-Neumann}
\forall \, n \in \N \, ,\quad \forall \, \mu=1,\dots,p \, ,\quad D_-^ku_{J+\mu}^n \, = \, 0 \, .
\end{equation}
The numerical boundary conditions can be understood as follows: the condition \eqref{schema-Neumann} determines $u_{J+1}^n$ in terms of 
interior values $u_{J+1-p}^n,\dots,u_J^n$ that are known in advance (take first $\mu=1$ in \eqref{schema-Neumann}). One then determines iteratively 
$u_{J+2}^n,\dots,u_{J+p}^n$ as when one solves a lower triangular linear system. As for $r$ and $p$, it is understood that $k$ is fixed while $J$ is 
large and satisfies at least $J+1-k \ge 0$ so that the ghost cell values $u_{J+1}^n,\dots,u_{J+p}^n$ are all determined from \eqref{schema-Neumann} 
by using the interior values $u_0^n,\dots,u_J^n$.

The numerical scheme \eqref{schema-int}, \eqref{schema-Dirichlet}, \eqref{schema-Neumann} is initialized by considering:
$$
\forall \, j=0,\dots,J \, ,\quad u_j^0 \, := \, \dfrac{1}{\Delta x} \, \int_{x_j}^{x_{j+1}} \, f(x) \, {\rm d}x \, ,
$$
where we recall that $f$ stands for the initial condition in \eqref{transport} and $x_0=0$, $x_{J+1}=L$. The ghost cell values $u_{-r}^0,\dots,u_{-1}^0$ 
and $u_{J+1}^0,\dots,u_{J+p}^0$ at the initial time $n=0$ are determined by \eqref{schema-Dirichlet} and \eqref{schema-Neumann}.
\bigskip

There are two ways to consider and implement the numerical scheme \eqref{schema-int}, \eqref{schema-Dirichlet}, \eqref{schema-Neumann}. 
One can either consider it as a time iteration on the vector $(u_0^n,\dots,u_J^n) \in \R^{J+1}$ and write it as:
\begin{equation}
\label{schema}
\forall \, n \in \N \, ,\quad \begin{pmatrix}
u_0^{n+1} \\
\vdots \\
u_J^{n+1} \end{pmatrix} \, = \, A \, \begin{pmatrix}
u_0^n \\
\vdots \\
u_J^n \end{pmatrix} \, ,
\end{equation}
for a convenient square matrix $A \in \mathscr{M}_{J+1}(\R)$. Or one can consider \eqref{schema-int}, \eqref{schema-Dirichlet}, 
\eqref{schema-Neumann} as a time iteration on the extended vector $(u_{-r}^n,\dots,u_{J+p}^n) \in \R^{J+p+r+1}$ but the set of all possible 
initial conditions is submitted to the restrictions imposed by \eqref{schema-Dirichlet} or \eqref{schema-Neumann} so the stability problem is 
less easy to rewrite in that framework (though the implementation might be easier). We shall therefore consider here the formulation \eqref{schema}, 
which amounts to considering all ghost cell values as auxiliary data that are necessary in the iteration process but that are not meaningful in 
the stability analysis.
\bigskip

The main stability problem that we investigate here is inspired from \cite{Benoit} and amounts to determining whether the iteration matrix 
$A$ is power bounded:
\begin{equation}
\label{stability}
\sup_{n \in \N} \, \| A^n \| \, < \, +\infty \, ,
\end{equation}
where $\| \cdot \|$ is, at this stage, \emph{any} norm on $\mathscr{M}_{J+1}(\R)$ since all norms are equivalent on that space. As is well-known, 
a necessary condition for \eqref{stability} to hold is that the spectral radius of $A$ should not be larger than $1$. We do not discuss here the 
choice of the norm and the dependence on the (large) dimension $J$ even though this issue would be extremely meaningful in a discussion 
about convergence estimates. Our primary focus here is to determine whether the spectral radius of $A$ can exceed $1$.
\bigskip

We first review some known stability estimates for the numerical scheme \eqref{schema-int}, \eqref{schema-Dirichlet}, \eqref{schema-Neumann} 
and the two related problems on a half-line. We then present some numerical evidence for the existence of exponentially growing numerical wave 
packets that may occur even for the very first examples $k=1$ (first order extrapolation) or $k=2$ (second order extrapolation). These examples 
indicate that the stability condition exhibited in \cite{Benoit} is not automatically satisfied in the case of the Dirichlet and Neumann condition when 
the stencil of the numerical scheme is arbitrary.

\section{A reminder on Dirichlet and Neumann numerical boundary conditions}
\label{section2}

\subsection{Three point schemes}

We first recall why the case of three point schemes and $k=1$ (first order extrapolation) can be easily handled. We consider the case $p=r=1$, 
so that the three coefficients $a_{-1},a_0,a_1$ satisfy \eqref{consistency}. The analysis below even allows $p$ (and therefore $a_1$) to be zero. 
We can equivalently rewrite the iteration \eqref{schema-int} as:
\begin{equation}
\label{schema-3pts}
\forall \, j=0,\dots,J \, ,\quad 
u_j^{n+1} \, = \, u_j^n \, - \, \dfrac{\lambda \, a}{2} \, (u_{j+1}^n-u_{j-1}^n) \, + \, \dfrac{\nu}{2} \, (u_{j+1}^n-2\, u_j^n+u_{j-1}^n) \, ,
\end{equation}
where $\nu \in \R$ is a parameter that may depend on $\lambda$ and $a$. Typical choices are:
\begin{itemize}
 \item $\nu=1$ ; one then obtains the so-called Lax-Friedrichs scheme.
 \item $\nu=\lambda \, a$ ; one then obtains the so-called upwind scheme (for which $p=0$ and $a_1=0$).
 \item $\nu=(\lambda \, a)^2$ ; one then obtains the so-called Lax-Wendroff scheme.
\end{itemize}

The interior scheme \eqref{schema-3pts} is combined with the Dirichlet condition \eqref{schema-Dirichlet} on the left of the interval, that is, 
$u_{-1}^n=0$, and the first order extrapolation (or Neumann) condition \eqref{schema-Neumann} with $k=1$ on the right of the interval, that 
is, $u_{J+1}^n=u_J^n$. Setting:
$$
a_{-1} \, := \, \dfrac{\lambda \, a+\nu}{2} \, ,\quad a_0 \, := \, 1-\nu \, ,\quad a_1 \, := \, \dfrac{-\lambda \, a+\nu}{2} \, ,
$$
the associated iteration matrix $A$ in \eqref{schema} reads:
\begin{equation}
\label{A-3pts}
A \, = \, \begin{pmatrix}
a_0 & a_1 & 0 & \cdots & \cdots & \cdots & 0 \\
a_{-1} & a_0 & a_1 & \ddots & & & \vdots \\
0 & \ddots & \ddots & \ddots & \ddots & & \vdots \\
\vdots & \ddots & \ddots  & \ddots & \ddots & \ddots & \vdots \\
\vdots & & \ddots & a_{-1} & a_0 & a_1 & 0 \\
\vdots & & & 0 & a_{-1} & a_0 & a_1 \\
0 & \cdots & \cdots & \cdots & 0 & a_{-1} & a_0+a_1 \end{pmatrix} \in \mathscr{M}_{J+1}(\R) \, .
\end{equation}
A straightforward result is the following:

\begin{lemma}
\label{lem1}
Let the parameters in \eqref{schema-3pts} satisfy:
$$
\lambda \, a \in [0,1] \, ,\quad \nu \in [(\lambda \, a)^2,1] \, .
$$
Then the matrix $A$ in \eqref{A-3pts} is power bounded.
\end{lemma}

This is, to some extent, the most favorable case where the coupling of Dirichlet and Neumann conditions at the inflow and outflow boundaries 
yields a straightforward stability estimate.

\begin{proof}[Proof of Lemma \ref{lem1}]
We introduce the standard $\ell^2$-norm on $\R^{J+1}$:
$$
\forall \, X \in \R^{J+1} \, ,\quad |X|^2 \, := \, \sum_{j=0}^J \, X_j^2 \, ,
$$
and let $\| \cdot \|$ denote the associated matrix norm on $\mathscr{M}_{J+1}(\R)$. The argument below will yield $\| A \| \le 1$, showing in 
particular that $A$ is power bounded since $\| \cdot \|$ is a matrix norm. The bound for the powers of $A$ is even uniform here with respect 
to the dimension $J$.

We consider a vector $(u_0,\dots,u_J) \in \R^{J+1}$ and define the ghost cell values $u_{-1}:=0$ and $u_{J+1}:=u_J$. We then define the vector 
$(v_0,\dots,v_J) \in \R^{J+1}$ as:
$$
\begin{pmatrix}
v_0 \\
\vdots \\
v_J \end{pmatrix} \, = \, A \, \begin{pmatrix}
u_0 \\
\vdots \\
u_J \end{pmatrix} \, ,
$$
so that we equivalently have:
$$
\forall \, j=0,\dots,J \, ,\quad v_j \, = \, a_{-1} \, u_{j-1} +a_0 \, u_j +a_1 \, u_{j+1} \, .
$$
In order to show that the norm of $A$ is less than $1$, it is sufficient to show the following inequality:
\begin{equation}
\label{ineg-lem1}
\sum_{j=0}^J \, v_j^2 \, \le \, \sum_{j=0}^J \, u_j^2 \, .
\end{equation}
Furthermore, from the expression of $v_j$ and straightforward algebraic manipulations, we find the relation\footnote{More general decompositions 
of the same kind can be found in \cite{CL-KRM}.}:
\begin{align*}
v_j^2 \, - \, u_j^2 \, =& \, -\dfrac{\nu-(\lambda \, a)^2}{2} \, \Big( (u_j-u_{j-1})^2+(u_{j+1}-u_j)^2 \Big) 
+\dfrac{\nu^2-(\lambda \, a)^2}{4} \, (u_{j+1}-2 \, u_j+u_{j-1})^2 \\
& \, -(\lambda \, a) \, u_j^2 +\dfrac{\nu \, (1-\lambda \, a)}{2} \, (u_{j+1}-u_j)^2 +(\nu-\lambda \, a) \, u_j \, (u_{j+1}-u_j) \\
& \, +(\lambda \, a) \, u_{j-1}^2 -\dfrac{\nu \, (1-\lambda \, a)}{2} \, (u_j-u_{j-1})^2 -(\nu-\lambda \, a) \, u_{j-1} \, (u_j-u_{j-1}) \, ,
\end{align*}
where the two last lines on the right-hand side are telescopic with respect to $j$. We sum this relation with respect to $j$ from $0$ to $J$ and make 
use of the relations $u_{-1}=0$, $u_{J+1}=u_J$ to simplify the \emph{boundary} terms (that are obtained after summing the last two lines). We get:
\begin{align*}
\sum_{j=0}^J \, v_j^2 \, - \, \sum_{j=0}^J \, u_j^2 \, =& \, -\dfrac{\nu-(\lambda \, a)^2}{2} \,  \sum_{j=0}^J \, (u_j-u_{j-1})^2+(u_{j+1}-u_j)^2 \\
&+\dfrac{\nu^2-(\lambda \, a)^2}{4} \, \sum_{j=0}^J \, (u_{j+1}-2 \, u_j+u_{j-1})^2 -(\lambda \, a) \, u_J^2 -\dfrac{\nu \, (1-\lambda \, a)}{2} \, u_0^2 \, .
\end{align*}

From our assumption on $\lambda \, a$ and $\nu$, the very two last terms in the second line on the right-hand side are nonpositive, so we have:
\begin{align*}
\sum_{j=0}^J \, v_j^2 \, - \, \sum_{j=0}^J \, u_j^2 \, \le& \, -\dfrac{\nu-(\lambda \, a)^2}{2} \,  \sum_{j=0}^J \, (u_j-u_{j-1})^2+(u_{j+1}-u_j)^2 \\
&+\dfrac{\nu^2-(\lambda \, a)^2}{4} \, \sum_{j=0}^J \, (u_{j+1}-2 \, u_j+u_{j-1})^2 \, .
\end{align*}
The inequality \eqref{ineg-lem1} follows in a straightforward way in the case $\nu \in [(\lambda \, a)^2,\lambda \, a]$ for in that case the right-hand 
side is the sum of two nonpositive quantities. We thus now examine the final case $\nu \in [\lambda \, a,1]$, and we use the inequality\footnote{This 
is a mere consequence of the inequality $(a-b)^2 \le 2 \, a^2+2 \, b^2$.}:
$$
(u_{j+1}-2 \, u_j+u_{j-1})^2 \, \le \, 2 \, (u_j-u_{j-1})^2 + 2 \, (u_{j+1}-u_j)^2 \, ,
$$
to get:
$$
\sum_{j=0}^J \, v_j^2 \, - \, \sum_{j=0}^J \, u_j^2 \, = \, -\dfrac{\nu-\nu^2}{2} \, \sum_{j=0}^J \, (u_j-u_{j-1})^2+(u_{j+1}-u_j)^2 \le 0 \, ,
$$
which shows again the validity of \eqref{ineg-lem1}. The proof of Lemma \ref{lem1} is complete.
\end{proof}

\noindent One can wonder whether the power boundedness of $A$ is linked to our choice $k=1$. If we had chosen $k=2$ in \eqref{schema-Neumann}, 
which corresponds to a second order extrapolation at the outflow boundary, the corresponding matrix $A$ would have read:
\begin{equation*}
A \, = \, \begin{pmatrix}
a_0 & a_1 & 0 & \cdots & \cdots & \cdots & 0 \\
a_{-1} & a_0 & a_1 & \ddots & & & \vdots \\
0 & \ddots & \ddots & \ddots & \ddots & & \vdots \\
\vdots & \ddots & \ddots  & \ddots & \ddots & \ddots & \vdots \\
\vdots & & \ddots & a_{-1} & a_0 & a_1 & 0 \\
\vdots & & & 0 & a_{-1} & a_0 & a_1 \\
0 & \cdots & \cdots & \cdots & 0 & a_{-1}-a_1 & a_0+2\, a_1 \end{pmatrix} \in \mathscr{M}_{J+1}(\R) \, ,
\end{equation*}
and it is then a new problem to determine whether $A$ is power bounded (under the same restrictions on $\lambda \, a$ and $\nu$ or under 
more severe restrictions). It is actually shown in \cite{CL-note} that the above matrix $A$ is power bounded under the very same restrictions 
on $\lambda \, a$ and $\nu$ as in Lemma \ref{lem1}. The proof however relies on a suitable \emph{modification} of the $\ell^2$ norm close 
to the outflow boundary, which is in the same spirit as the analysis in \cite{strand}. This gives, once again, a bound for the powers of $A$ that 
is even uniform with respect to the dimension $J$ but the above argument does not extend in a straightforward way. We do not know whether 
such \emph{energy} arguments can be used for three point schemes and any a priori given value of $k$, but the analysis in \cite{Benoit} 
suggests that the corresponding matrix $A$ should be power bounded. This is left to further study.

We discuss below in Section \ref{section3} why the simple scenario of Lemma \ref{lem1} does not extend to arbitrary stencils, but before 
that, we shall review some known facts on the numerical scheme \eqref{schema-int}, \eqref{schema-Dirichlet}, \eqref{schema-Neumann} for 
arbitrary $r,p$ and $k$.

\subsection{The general case}

This section summarizes the main results of \cite{CL-KRM}. In addition to \eqref{consistency}, we shall from now on assume that the coefficients 
$a_\ell$ satisfy the following (von Neumann) stability condition.

\begin{assumption}
\label{hyp:vonNeumann}
There holds:
$$
\sup_{\theta \in \R} \left| \sum_{\ell=-r}^p \, a_\ell \, {\rm e}^{\mbi \, \ell \, \theta} \right| \, = \, 1 \, .
$$
\end{assumption}

Under Assumption \ref{hyp:vonNeumann}, it is known that the convolution operator:
\begin{equation}
\label{operator-T}
\mathscr{T} \, : \, (u_j)_{j \in \Z} \in \ell^2(\Z;\R) \quad \longmapsto \quad \left( \sum_{\ell=-r}^p \, a_\ell \, u_{j+\ell} \right)_{j \in \Z} \in \ell^2(\Z;\R) 
\end{equation}
has norm $1$ and is therefore power bounded. Furthermore, the one-sided problem with Dirichlet boundary conditions on the left is also known 
to be a contraction. Namely, it is also known (see \cite{Wu,CG}) that the solution to the numerical scheme:
\begin{equation}
\label{schema-aux-1}
\begin{cases}
u_j^{n+1} \, = \, \sum_{\ell=-r}^p \, a_\ell \, u_{j+\ell}^n \, , & j \in \N \, ,\quad n \in \N \, ,\\
u_\mu^n \, = \, 0 \, , & \mu=-r,\dots,-1 \, ,\quad n \in \N \, ,
\end{cases}
\end{equation}
satisfies:
$$
\forall \, n \in \N \, ,\quad \sum_{j \in \N} \, (u_j^{n+1})^2 \, \le \, \sum_{j \in \N} \, (u_j^n)^2 \, .
$$
This contraction property holds because the time iteration in \eqref{schema-aux-1} can be written under the form $\pi \, \mathscr{T} \, \pi$ 
where $\pi$ denotes the \emph{orthogonal} projection in $\ell^2(\Z;\R)$ on the sequences that are supported on $\N$, and $\mathscr{T}$ is 
the pure convolution operator given in \eqref{operator-T} whose norm equals $1$. We should therefore always keep in mind that prescribing 
the Dirichlet boundary condition makes the $\ell^2$ norm decrease since it acts as an orthogonal projection and therefore preserves stability.

On the other hand, the outflow problem with extrapolation boundary condition:
\begin{equation}
\label{schema-aux-2}
\begin{cases}
u_j^{n+1} \, = \, \sum_{\ell=-r}^p \, a_\ell \, u_{j+\ell}^n \, , & j \le J \, ,\quad n \in \N \, ,\\
D_-^k u_{J+\mu}^n \, = \, 0 \, , & \mu=1,\dots,p \, ,\quad n \in \N \, ,
\end{cases}
\end{equation}
has also been analyzed in \cite{CL-KRM} where it is shown (under Assumption \ref{hyp:vonNeumann} and the consistency conditions 
\eqref{consistency}) that there exists a constant $C$ (that only depends on the given extrapolation order $k$) such that, independently 
of the initial condition, there holds:
$$
\forall \, n \in \N \, ,\quad \sum_{j \le J} \, (u_j^n)^2 \, \le \, C \, \sum_{j \le J} \, (u_j^0)^2 \, ,
$$
for any solution to \eqref{schema-aux-2}. In other words, the time evolution operator in \eqref{schema-aux-2} is power bounded on $\ell^2 
((-\infty,J];\R)$, even though it is not necessarily a contraction for the $\ell^2$ norm.

In view of the above two results, it is only the coupling between the Dirichlet and the Neumann condition that may create a large time 
exponential instability. Namely, the combination of the above two stability results (the one for Dirichlet and the one for the extrapolation 
condition) implies the following estimate for the matrix $A$ in \eqref{schema} (see the main result in \cite{CL-KRM}):
$$
\forall \, n \in \N \, ,\quad \| A^n \| \, \le \, C_0 \, {\rm e}^{n \, C_0/J} \, ,
$$
where the constant $C_0$ is independent of $n$ and $J$, and $\| \cdot \|$ is the matrix norm associated with the $\ell^2$ norm on vectors. 
In particular, the spectral radius of $A$ satisfies a bound of the form (here and from now on $\rho$ stands for the spectral radius of a matrix):
$$
\rho(A) \, \le \, 1 +\dfrac{C}{J} \, ,
$$
where $C$ is a constant that does not depend on $J$. Hence, if $A$ admits an unstable eigenvalue of modulus larger than $1$, it can not 
be ``too large''. Our goal is to make explicit whether such unstable eigenvalues do indeed occur.

\section{Examples of large time amplification}
\label{section3}

\subsection{The basic principles}

We make use of the theory of oscillating wave packets and shall use many times below the concept of \emph{group velocity}. This notion is 
discussed in details in \cite{Trefethen1} in the context of finite difference schemes and we refer to that reference for a detailed presentation. 
A convenient reference for the same notion in the context of partial differential equations is \cite{Whitham}. The application of these notions 
to numerical boundary conditions is the purpose of \cite{Trefethen2} and we shall also use the analysis of \cite{Trefethen2} as a black box 
in our (formal) arguments below. Namely, from Assumption \ref{hyp:vonNeumann}, we know that there are some values of $\theta \in \R$ 
such that the so-called \emph{amplification factor}:
$$
\sum_{\ell=-r}^p \, a_\ell \, {\rm e}^{\mbi \, \ell \, \theta}
$$
has modulus $1$. For instance, the consistency conditions \eqref{consistency} show that this is the case at $\theta=0$. From the von Neumann 
condition (that is, Assumption \ref{hyp:vonNeumann}), such modes are the only ones that are not exponentially damped by the numerical scheme. 
Let $\underline{\theta}$ be such a real number and let us set:
$$
\underline{z} \, := \, \sum_{\ell=-r}^p \, a_\ell \, {\rm e}^{\mbi \, \ell \, \underline{\theta}} \in \cercle \, .
$$
Then the plane wave:
$$
(n,j) \in \N \times \Z \, \longmapsto \, \underline{z}^n \, {\rm e}^{\mbi \, j \, \underline{\theta}} \, ,
$$
is a solution to the numerical scheme \eqref{schema-int}, at least far from the boundaries. Writing $j \, \underline{\theta}=(j \, \Delta x) \, 
\underline{\theta}/\Delta x$, we can view the mesh width $\Delta x$ as a small wavelength parameter. The theory developed in \cite{Trefethen1} 
shows that slowly modulated highly oscillating wave packets of the form:
\begin{equation}
\label{wavepacket}
\Phi(j\, \Delta x-{\bf v}_g \, n \, \Delta t) \, \underline{z}^n \, {\rm e}^{\mbi \, j \, \underline{\theta}}
\end{equation}
may be propagated by the numerical scheme \eqref{schema-int} for a well-chosen group velocity ${\bf v}_g$. The (smooth) function $\Phi$ 
describes the \emph{envelope} of the oscillations. In the trivial case $\underline{\theta}=0$, $\underline{z}=1$, one recovers the propagation 
of a smooth profile by a stable and consistent approximation of the transport equation; the group velocity ${\bf v}_g$ equals the transport 
velocity $a$ of \eqref{transport} in that case. The above expression \eqref{wavepacket} does not exactly provide with a solution to \eqref{schema-int} 
but it gives its \emph{leading behavior} if one has in mind that the solution to \eqref{schema-int} has an asymptotic expansion with respect to 
the small parameter $\Delta x$. This is exactly the same kind of arguments as in the so-called WKB analysis of linear geometric optics 
\cite{Lax}.

The boundary conditions on either side of the interval, meaning either \eqref{schema-Dirichlet} or \eqref{schema-Neumann}, will now give rise 
to wave packet \emph{reflection} as follows:
\begin{itemize}
 \item We start at time $t=0$ with some slowly modulated highly oscillating wave packet \eqref{wavepacket} that is supported in the middle of 
 the computation interval $(0,L)$ (take for instance an envelope function $\Phi$ that is supported in $[L/3,2L/3]$).
 \item Assume that the numerical wave packet \eqref{wavepacket} has a positive group velocity ${\bf v}_g$, so that it will eventually reach the 
 right boundary $x=L$ at some positive time.
 \item When hitting the boundary, the wave packet acts as a forcing term in the Neumann condition \eqref{schema-Neumann} of the form 
 $\alpha_\mu \, \underline{z}^n$, $\mu=1,\dots,p$, where $\alpha_\mu$ is computed by applying the Neumann condition to the wave packet 
 itself. One should then determine whether the numerical scheme \eqref{schema-int} supports wave packets of the form:
$$
\underline{z}^n \, \kappa^j \, ,
$$
 with $|\kappa| \ge 1$ that will be reflected ``backwards''. In the limit case $|\kappa| = 1$, the associated group velocity of the wave packet should 
 be nonpositive so that the wave packet will propagate towards the left.
 \item In the case where in the previous reflection process, one of the group velocities is negative, the corresponding reflected wave packet will 
 eventually hit the left boundary $x=0$ and a similar reflection process will take place (the corresponding selection for the $\kappa$'s is now 
 $|\kappa| \le 1$, and in the limit case $|\kappa| = 1$, the associated group velocity of the wave packet should be nonnegative).
\end{itemize}

Several points should be kept in mind : if a reflection gives rise to several wave packets, they will each carry part of the energy of the solution. 
In order to make the $\ell^2$ norm of the whole solution increase, it is desirable to have a sequence of reflections such that after one reflection 
on the right and one on the left, a wave packet with a given frequency $\underline{\theta}$ and positive group velocity has its amplitude multiplied 
by a factor that has modulus larger than $1$. Unless some cancelations appear, the wave packet \eqref{wavepacket} will necessarily be generated 
after reflection on the left of the interval since it has a positive group velocity. Observe now that it takes $O(J)$ time steps to make the whole transport 
back and forth from left to right, so an amplification pattern as described above should lead to a solution whose $\ell^2$ norm grows at least like 
$\exp(c \, n/J)$ for some positive constant $c$. This is precisely the situation which we try to highlight here for, in that case, we expect the spectral 
radius of $A$ to be larger than $1$.

In what follows, we seek for such exponential amplifications by using as an initial condition a smooth profile, that is $\underline{\theta}=0$, 
$\underline{z}=1$. To make the above mechanism work, the numerical scheme should admit an oscillating wave packet:
$$
1^n \, \kappa^j
$$
with $|\kappa|=1$ and a \emph{negative} group velocity ($\kappa$ is necessarily not equal to $1$ for in that case the group velocity equals $a$ 
and is therefore positive). However, one should observe that when the Neumann condition \eqref{schema-Neumann} is applied to a smooth 
wave of the form $\Phi(j\, \Delta x-a \, n \, \Delta t)$, the corresponding error is $O(\Delta x^k)$ so the forcing term in \eqref{schema-Neumann} 
will have a very small reflection coefficient. Since the Dirichlet condition makes the $\ell^2$ norm decrease, this will not be sufficient to trigger 
an instability. We therefore need another \emph{auxiliary} wave packet that is associated with a \emph{nonzero} frequency and that has a 
\emph{positive} group velocity.

Let us therefore summarize what we need to display some two boundary large time instability :
\begin{enumerate}
 \item A finite difference scheme \eqref{schema-int} that satisfies the conditions \eqref{consistency} and the von Neumann stability condition 
 (Assumption \ref{hyp:vonNeumann}).
 \item The finite difference scheme should also admit one plane wave of the form $1^n \, \kappa_1^j$ that is associated with a negative group 
 velocity ${\bf v}_{g,1}$ and another plane wave of the form $1^n \, \kappa_2^j$ that is associated with a positive group velocity ${\bf v}_{g,2}$, 
 both $\kappa_1$ and $\kappa_2$ being distinct from $1$.
 \item If the reflection coefficient from the first wave to the second on the right boundary is denoted $\alpha_{1\to2}$, and if the reflection 
 coefficient from the second wave to the first on the right boundary is denoted $\beta_{2\to1}$, then we wish $|\alpha_{1\to2} \beta_{2\to1}|>1$.
\end{enumerate}

\subsection{An example with first order extrapolation}

In this first example, we have $r=p=7$. The scheme reads as in \eqref{schema-int} with the following choice for the coefficients:
\begin{subequations}\label{coeff1}
\begin{align}
a_{-7}&=\dfrac{20133}{704759} \, ,\quad & a_{-6}&=-\dfrac{30433}{894654} \, ,\quad & a_{-5}&=\dfrac{20476}{371655} \, ,
\quad & a_{-4}&=\dfrac{12703}{838076} \, ,\\
a_{-3}&=\dfrac{75599}{770583} \, ,\quad & a_{-2}&=\dfrac{31384}{945409} \, ,\quad & a_{-1}&=-\dfrac{4015}{85641} \, ,
\quad & a_0&=\dfrac{261251}{274289} \, ,\\
a_1&=\dfrac{53837}{928392} \, ,\quad & a_2&=-\dfrac{27478}{587215} \, ,\quad & a_3&=-\dfrac{54064}{714213} \, ,
\quad & a_4&=-\dfrac{27635}{674698} \, ,\\
a_5&=-\dfrac{31244}{798847} \, ,\quad & a_6&=\dfrac{23091}{711760} \, ,\quad & a_7&=\dfrac{1864}{178339} \, .
\quad & &
\end{align}
\end{subequations}
The corresponding amplification factor is depicted in Figure~\ref{fig:spectrum1}. We implement this numerical scheme with the Dirichlet boundary 
condition \eqref{schema-Dirichlet} on the left of the interval and the Neumann condition  \eqref{schema-Neumann} on the right  ($k=1$), 
that is:
$$
\forall \, n \in \N \, ,\quad u_{J+1}^n=\cdots=u_{J+7}^n=u_J^n \, .
$$

\begin{figure}\centering
\includegraphics[scale=0.45]{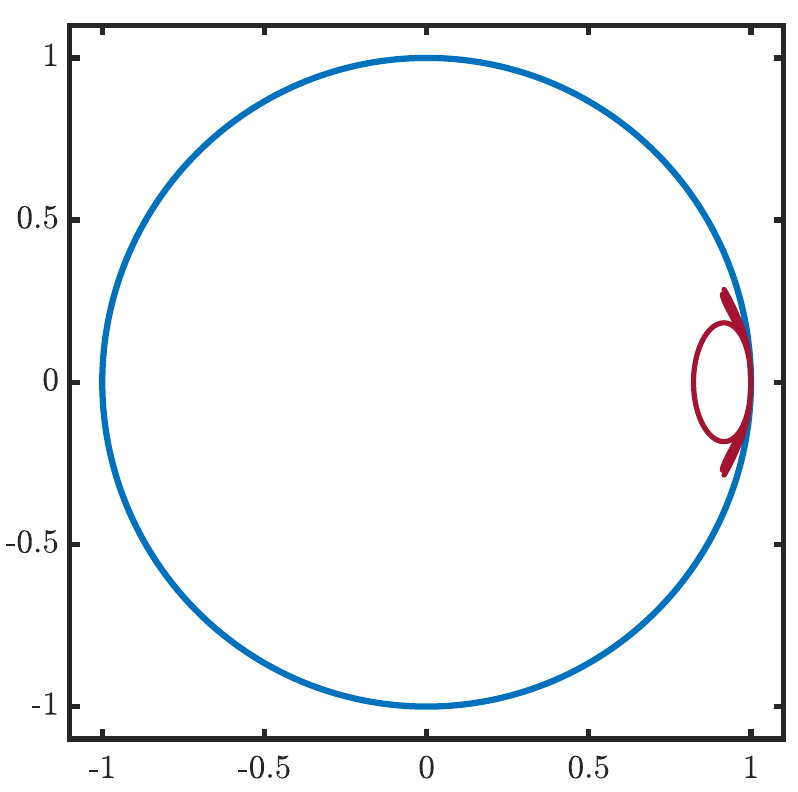}\hspace{1cm}
\includegraphics[scale=0.45]{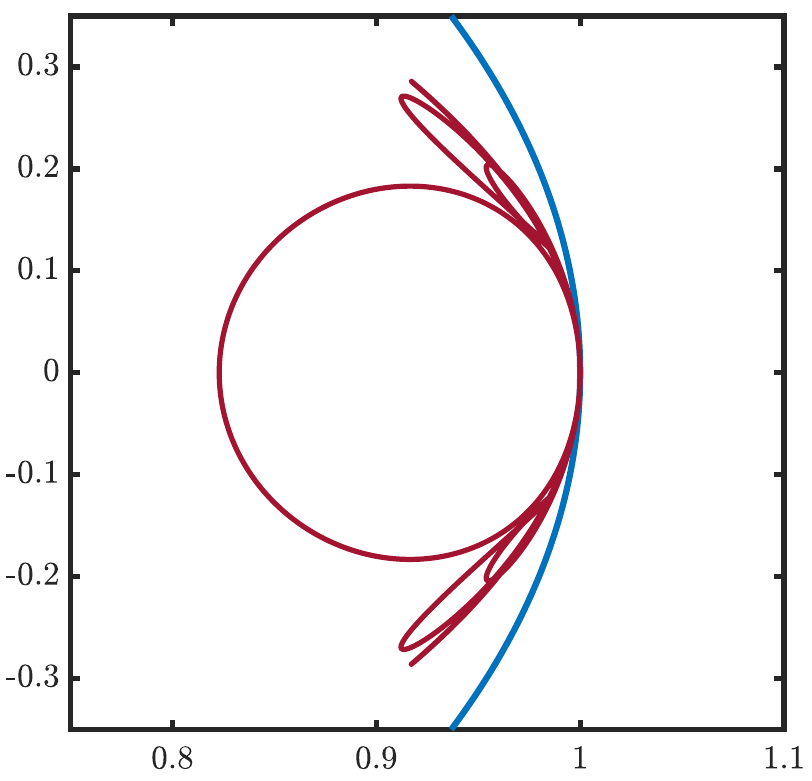}
\caption{Left: amplification factor (dark red curve) $\theta \mapsto \sum_{\ell=-r}^p \, a_\ell \, {\rm e}^{\mbi \, \ell \, \theta}$ with $r=p=7$ and coefficients 
$a_\ell$ given in \eqref{coeff1} within the unit circle $\mathbb{S}^1$ (blue curve). Right: zoom of the amplification factor near the tangency point at $z=1$.}
\label{fig:spectrum1}
\end{figure}

The above numerical scheme is devised in such a way that it supports the following wave packets with given frequencies and group velocities:
\begin{align*}
(\underline{\kappa}_1,{\bf v}_{g,1})&=(1,1) \, ,& & \\
(\underline{\kappa}_2,{\bf v}_{g,2})&=({\rm e}^{\mbi 0.288 \pi},-1) \, ,\quad &
(\underline{\kappa}_3,{\bf v}_{g,3})&=({\rm e}^{-\mbi 0.288 \pi},-1) \, ,\\
(\underline{\kappa}_4,{\bf v}_{g,4})&=({\rm e}^{\mbi 0.82 \pi},1) \, ,\quad &
(\underline{\kappa}_5,{\bf v}_{g,5})&=({\rm e}^{-\mbi 0.82 \pi},1) \, .
\end{align*}
For all the above five values of $\kappa$, the associated amplification factor $z$ equals $1$ so all wave packets are associated with the 
\emph{same} time frequency, which allows for boundary reflection back and forth. If one starts from a smooth initial condition, the wave packet 
associated with the zero frequency $(\underline{\kappa}_1,{\bf v}_{g,1})$ will first travel to the right. Its reflection on the right boundary of the 
interval will give rise to the two wave packets associated with $(\underline{\kappa}_2,{\bf v}_{g,2})$ and $(\underline{\kappa}_3,{\bf v}_{g,3})$, 
both having small amplitudes since any smooth wave satisfies the Neumann condition with consistency error $O(\Delta x)$. When those reflected 
wave packets will hit the left boundary of the interval, the wave packet associated with $(\underline{\kappa}_1,{\bf v}_{g,1})$ will be generated 
again together with the wave packet associated with $(\underline{\kappa}_4,{\bf v}_{g,4})$ and $(\underline{\kappa}_5,{\bf v}_{g,5})$.

\begin{figure}\centering
\includegraphics[scale=0.45]{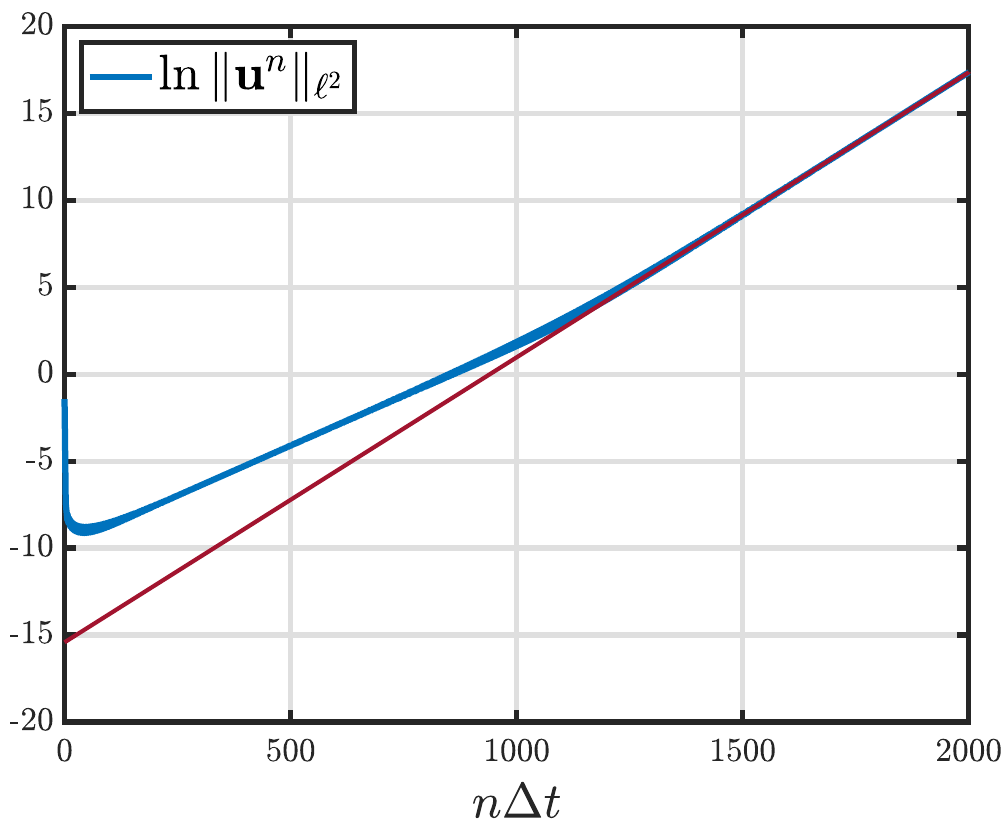}
\caption{Evolution of $\ln \|\mathbf{u}^n\|_{\ell^2}$ as a function of $n\Delta t$ for the solution $\mathbf{u}^n=(u_j^n)_{j=0,\cdots,J}$ of the numerical 
scheme starting from a smooth initial condition with the Dirichlet boundary  condition \eqref{schema-Dirichlet} on the left of the interval and the Neumann 
condition  \eqref{schema-Neumann} on the right  ($k=1$) together with the coefficients given in \eqref{coeff1}. Here, we have set $J=994$ with  $\Delta t 
=\Delta x=1/(J+1)$.}
\label{fig:norm1smooth}
\end{figure}

The amplification here will arise from the two pairs $(\underline{\kappa}_2,\underline{\kappa}_3)$ and $(\underline{\kappa}_4,\underline{\kappa}_5)$ 
of oscillating wave packets propagating back and forth from $0$ to $1$. We emphasize that the amplification is rather low, and we refer to 
Figure~\ref{fig:norm1smooth} for an illustration of the evolution of the logarithmic quantity $\ln \|\mathbf{u}^n\|_{\ell^2}$ for the solution of the 
numerical scheme $\mathbf{u}^n=(u_j^n)_{j=0,\cdots,J}$ as a function of $n\Delta t$ when the numerical scheme is initialized with the following 
sequence:
\bqs
\forall \, j=0,\cdots, J\,, \quad u_j^0 = f(x_j),\, \text{ with } \, f(x):=\exp\left(-50\left(x-\frac{1}{2}\right)^2\right)\,.
\eqs
Here, we have used the following $\ell^2$ norm:
\bqs
\|\mathbf{u}^n\|_{\ell^2}:= \left( \Delta x \, \sum_{j=0}^J \left|u_j^n\right|^2\right)^{1/2}.
\eqs
We observe, as expected, that it takes $O(1/\Delta x)$ amount of time for the instability to take off. The numerically computed slope of the linear 
regression of $\ln \|\mathbf{u}^n\|_{\ell^2}$ as a function of $n\Delta t$ is $1.63818 \times 10^{-2}$ which compares very well to:
\bqs
\frac{\rho(A)-1}{\Delta x} \sim 1.63995 \times 10^{-2},
\eqs
where we denoted by $\rho(A)$ the spectral radius of the matrix $A$ of the numerical scheme. This phenomenon can be expected from the so-called 
power method for computing the largest eigenvalue of a matrix.

In order to better illustrate the reflection of the wave packets, Figure~\ref{fig:norm1} shows the evolution of the solution of the numerical scheme 
initialized with the following (slowly modulated, highly oscillating) sequence:
\bqs
\forall \, j=0,\cdots, J\,, \quad u_j^0 = f(x_j),\, \text{ with } \, 
f(x):=\cos\left(\frac{0.82\pi}{\Delta x}\left(x-\frac{1}{2}\right)\right)\exp\left(-50\left(x-\frac{1}{2}\right)^2\right)\,,
\eqs
which corresponds to the linear superposition of the slowly modulated highly oscillating wave packets $(\underline{\kappa}_4,{\bf v}_{g,4})$ and 
$(\underline{\kappa}_5,{\bf v}_{g,5})$. As expected, we observe a reflection of the wave packets with an amplification over time.

\begin{figure}\centering
\includegraphics[scale=0.45]{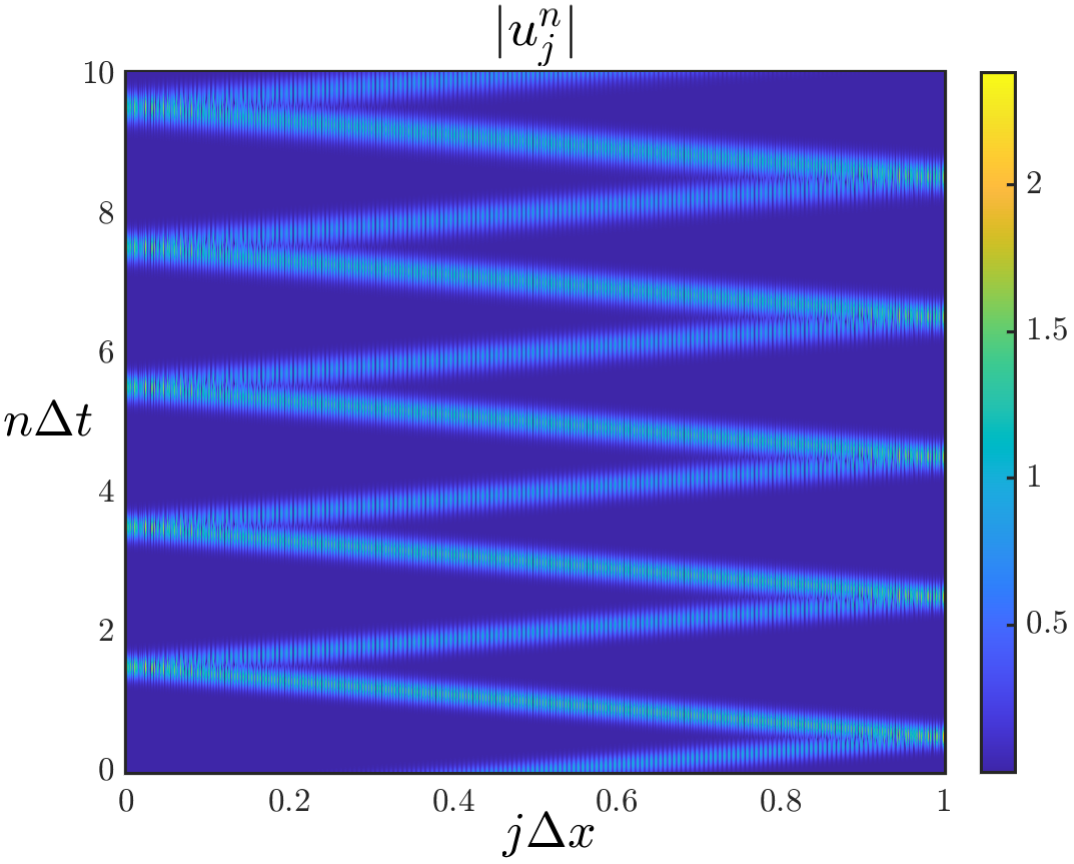}\hspace{1cm}
\includegraphics[scale=0.45]{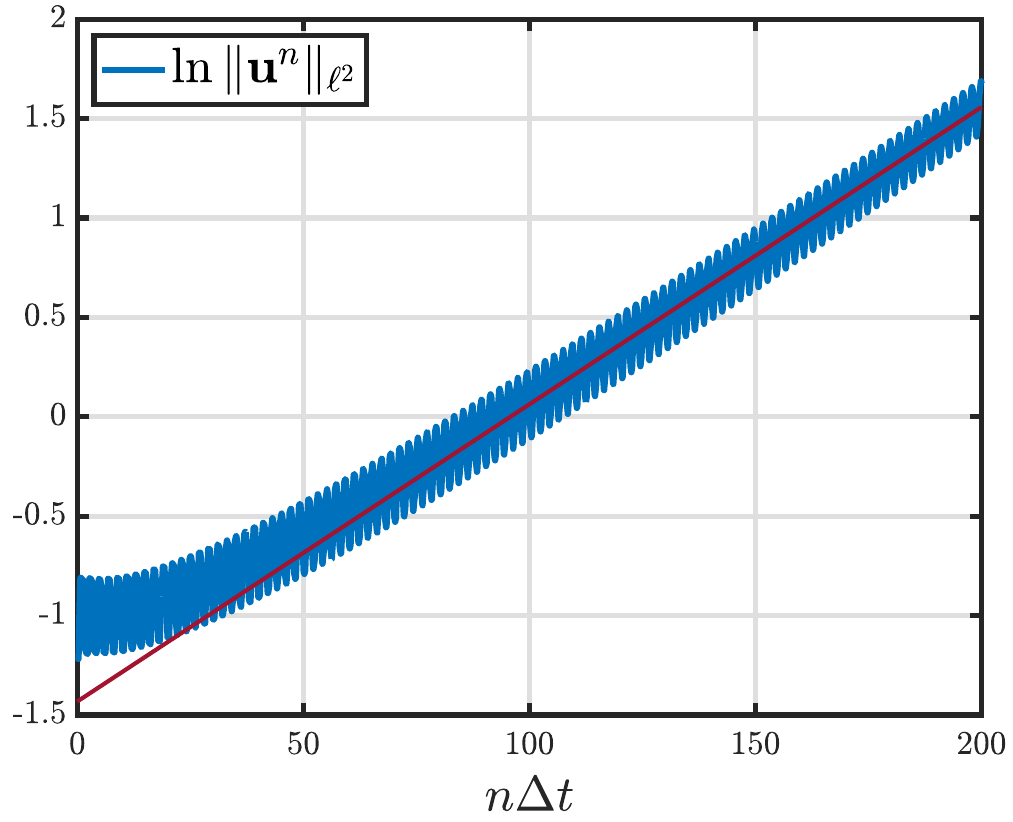}
\caption{Left: space-time evolution of $|u_j^n|$, solution of the numerical scheme with the Dirichlet boundary 
condition \eqref{schema-Dirichlet} on the left of the interval and the Neumann condition  \eqref{schema-Neumann} on the right  ($k=1$) together with 
the coefficients given in \eqref{coeff1}. Right: evolution of $\ln \|\mathbf{u}^n\|_{\ell^2}$ as a function of $n\Delta t$. Here, we have set $J=994$ with 
$\Delta t=\Delta x=1/(J+1)$.}
\label{fig:norm1}
\end{figure}

\subsection{An example with second order extrapolation}

In this second example, we still have $r=p=7$. The scheme reads as in \eqref{schema-int} with the following choice for the coefficients:
\begin{subequations}\label{coeff2}
\begin{align}
a_{-7}&=\dfrac{17151}{869039} \, ,\quad & a_{-6}&=-\dfrac{9591}{473236} \, ,\quad & a_{-5}&=\dfrac{55269}{926798} \, ,
\quad & a_{-4}&=-\dfrac{1854}{92119} \, ,\\
a_{-3}&=\dfrac{8983}{71254} \, ,\quad & a_{-2}&=\dfrac{32579}{620922} \, ,\quad & a_{-1}&=-\dfrac{47474}{813983} \, ,
\quad & a_0&=\dfrac{739673}{796040} \, ,\\
a_1&=\dfrac{2966}{34841} \, ,\quad & a_2&=-\dfrac{19830}{397889} \, ,\quad & a_3&=-\dfrac{71152}{650153} \, ,
\quad & a_4&=-\dfrac{21338}{716071} \, ,\\
a_5&=-\dfrac{10189}{548431} \, ,\quad & a_6&=\dfrac{19029}{761263} \, ,\quad & a_7&=\dfrac{8820}{964529} \, .
\quad & &
\end{align}
\end{subequations}
The corresponding amplification factor is depicted in Figure~\ref{fig:spectrum2}. We implement this numerical scheme with the Dirichlet boundary 
condition \eqref{schema-Dirichlet} on the left of the interval and the second order extrapolation condition \eqref{schema-Neumann} 
on the right ($k=2$), that is:
$$
\forall \, n \in \N \, ,\quad u_{J+1}^n= 2 \, u_J^n -u_{J-1}^n \, ,\quad u_{J+2}^n= 2 \, u_{J+1}^n -u_J^n \, \dots
$$
which can be further simplified into:
$$
\forall \, n\in \N \, ,\quad \forall \, \mu =1,\cdots, p\,, \quad u_{J+\mu}^n= (\mu+1) \, u_J^n - \mu \, u_{J-1}^n \, .
$$

\begin{figure}\centering
\includegraphics[scale=0.45]{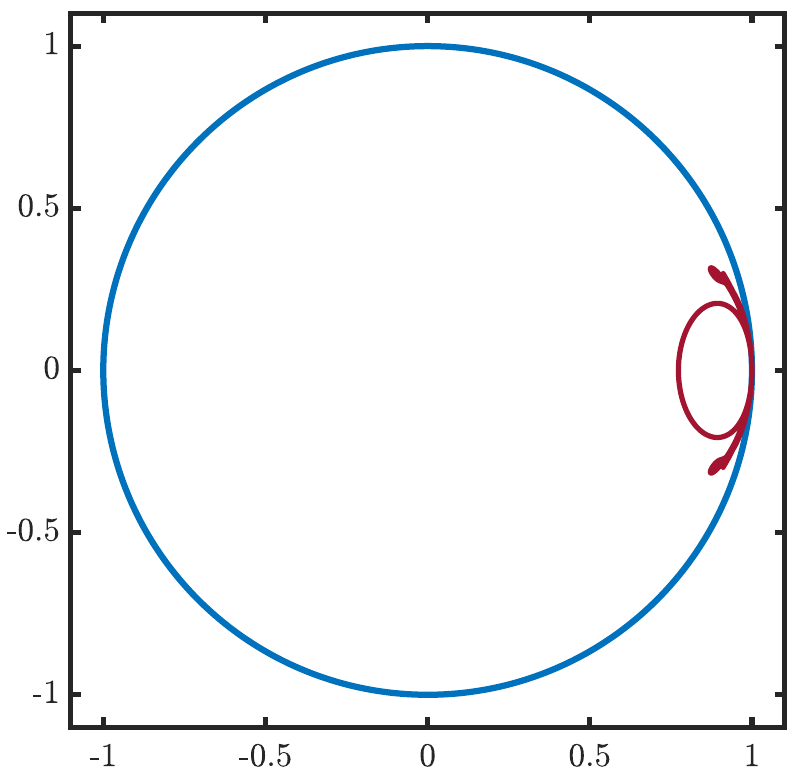}\hspace{1cm}
\includegraphics[scale=0.45]{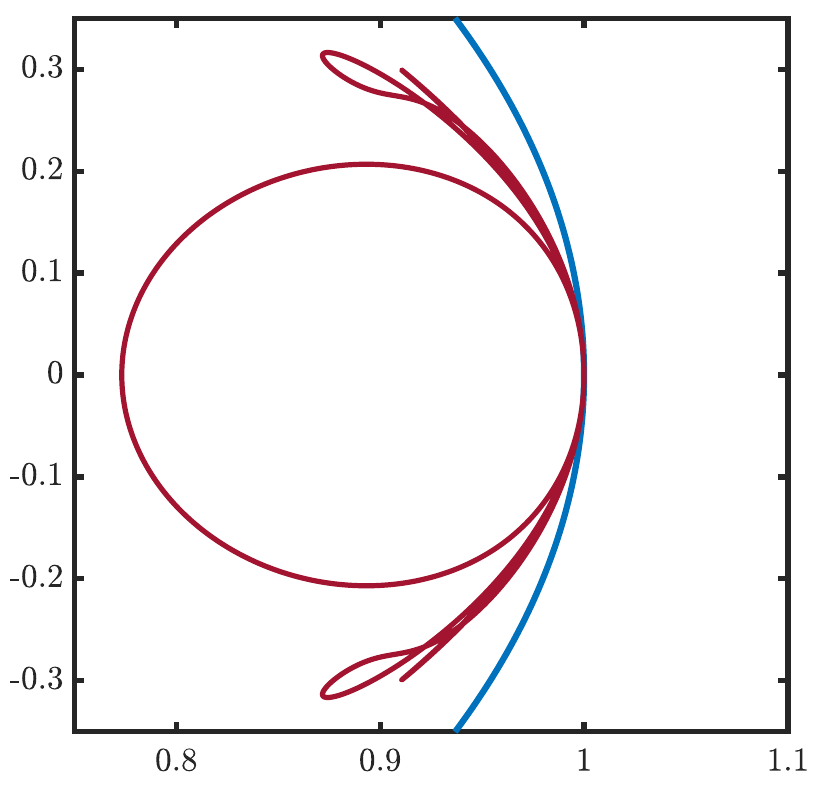}
\caption{Left: amplification factor (dark red curve) $\theta \mapsto \sum_{\ell=-r}^p \, a_\ell \, {\rm e}^{\mbi \, \ell \, \theta}$ with $r=p=7$ and 
coefficients $a_\ell$ given in \eqref{coeff2} within the unit circle $\mathbb{S}^1$ (blue curve). Right: zoom of the amplification factor near the 
tangency point at $z=1$.}
\label{fig:spectrum2}
\end{figure}

The above numerical scheme is devised in such a way that it supports the following wave packets with given frequencies and group velocities:
\begin{align*}
(\underline{\kappa}_1,{\bf v}_{g,1})&=(1,1) \, ,& & \\
(\underline{\kappa}_2,{\bf v}_{g,2})&=({\rm e}^{\mbi 0.3 \pi},-1) \, ,\quad &
(\underline{\kappa}_3,{\bf v}_{g,3})&=({\rm e}^{-\mbi 0.3 \pi},-1) \, ,\\
(\underline{\kappa}_4,{\bf v}_{g,4})&=({\rm e}^{\mbi 0.8 \pi},1) \, ,\quad &
(\underline{\kappa}_5,{\bf v}_{g,5})&=({\rm e}^{-\mbi 0.8 \pi},1) \, .
\end{align*}
For all the above five values of $\kappa$, the associated amplification factor $z$ equals $1$. The wave packet generation will thus be entirely 
similar as in our first example. Only the numerical values are different.

The amplification here will arise from the two pairs of oscillating wave packets propagating back and forth from $0$ to $1$. We emphasize 
that the amplification is stronger here than for the first order extrapolation. As could be expected, the reason is that the spectral radius of $A$ 
will be larger.

Figure~\ref{fig:norm2} shows the evolution of the solution of the numerical scheme initialized with the following (slowly modulated, highly oscillating) 
sequence:
\bqs
\forall \, j=0,\cdots, J\,, \quad u_j^0 = f(x_j),\, \text{ with } \, 
f(x):=\cos\left(\frac{0.8\pi}{\Delta x}\left(x-\frac{1}{2}\right)\right)\exp\left(-50\left(x-\frac{1}{2}\right)^2\right)\,,
\eqs
which corresponds to the linear superposition of the slowly modulated highly oscillating wave packets $(\underline{\kappa}_4,{\bf v}_{g,4})$ and 
$(\underline{\kappa}_5,{\bf v}_{g,5})$. We observe, as expected, a reflection of the wave packets with an amplification over time. The numerically 
computed slope of the linear regression of $\ln \|\mathbf{u}^n\|_{\ell^2}$ as a function of $n\Delta t$ is $3.15834 \times 10^{-1}$ which compares 
very well to the computed value:
\bqs
\frac{\rho(A)-1}{\Delta x} \sim 3.15884 \times 10^{-1},
\eqs
where we denoted once again by $\rho(A)$ the spectral radius of the matrix $A$ of the numerical scheme. The amplification is indeed much stronger 
in the present case.

\begin{figure}\centering
\includegraphics[scale=0.45]{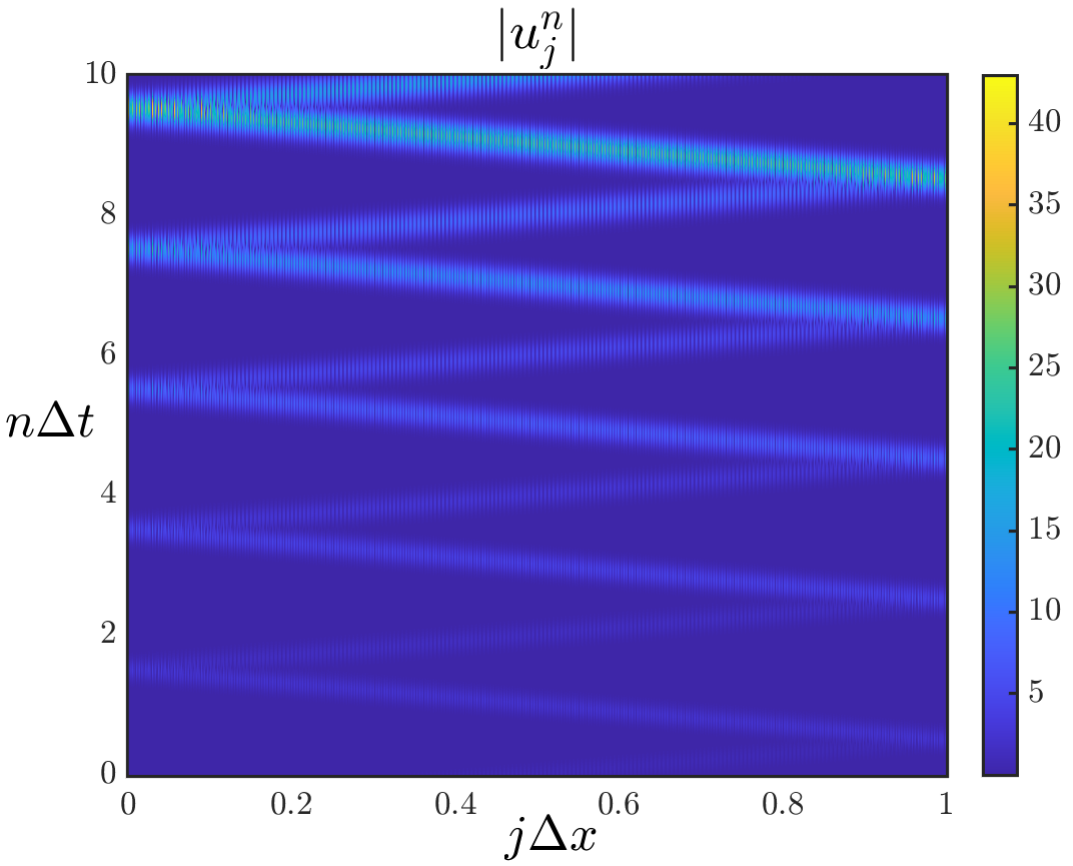}\hspace{1cm}
\includegraphics[scale=0.45]{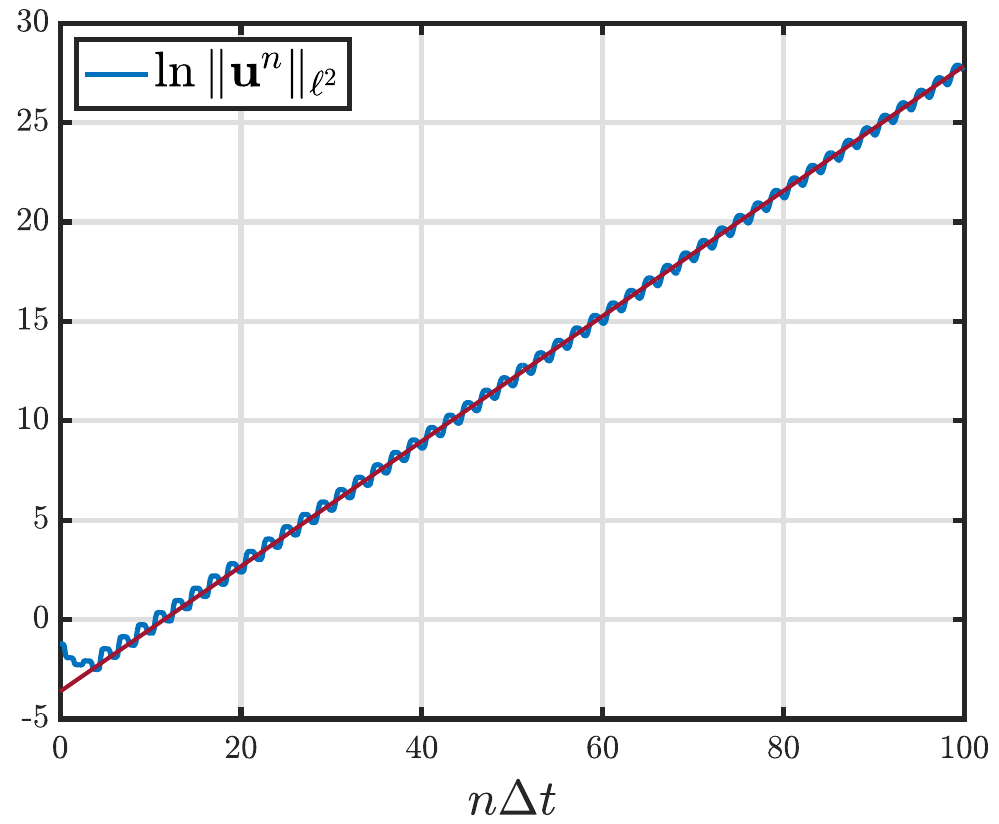}
\caption{Left: space-time evolution of $|u_j^n|$, solution of the numerical scheme with the Dirichlet boundary 
condition \eqref{schema-Dirichlet} on the left of the interval and the Neumann condition  \eqref{schema-Neumann} on the right  ($k=2$) together with 
the coefficients given in \eqref{coeff2}. Right: evolution of $\ln \|\mathbf{u}^n\|_{\ell^2}$ as a function of $n\Delta t$. Here, we have set $J=1000$ with 
$\Delta t=\Delta x=1/(J+1)$.}
\label{fig:norm2}
\end{figure}

\subsection{Discussion}

The above examples are meant to show that even though Dirichlet and Neumann numerical boundary conditions seem perfectly legitimate, 
large time stability issues may be an issue, even for one-step stable, explicit finite difference approximations of the transport equation. Still, 
the instability may be difficult to capture and it is quite unstable with respect to the various parameters, including the number of points $J$. 
If a large time instability occurs for some integer $J$, it may very well be that for $J+1$ the instability disappears. A connection with the 
spectral analysis of large Toeplitz matrices should be made.

Even though the reflection instability pattern is not hard to understand, there is no guarantee that any finite difference approximation with 
enough oscillating wave packets and appropriate group velocities will lead to an instability. To some extent, we have been lucky enough 
to obtain examples for $k=1$ and $k=2$.

Our main message is that the above examples show that the stability criterion exhibited in \cite{Benoit} may fail to be satisfied in several 
situations. This motivates setting numerical routines for trying to verify this criterion on high order schemes and more elaborate numerical 
boundary conditions.

\bibliographystyle{chicago}
\bibliography{Note}
\end{document}